\documentclass[11pt]{article}
\usepackage[a4paper,margin=1.15in]{geometry}
\usepackage[T1]{fontenc}
\usepackage[utf8]{inputenc}
\usepackage{amsmath,amssymb,amsthm,mathtools}
\usepackage{microtype}
\usepackage{enumitem}
\usepackage[numbers,sort&compress]{natbib}
\usepackage[colorlinks=true,linkcolor=blue,citecolor=blue,urlcolor=blue]{hyperref}
\usepackage[nameinlink,noabbrev]{cleveref}
\mathtoolsset{showonlyrefs}

\newtheorem{theorem}{Theorem}[section]
\newtheorem{lemma}[theorem]{Lemma}

\theoremstyle{definition}

\newcommand{\R}{\mathbb R}
\newcommand{\E}{\mathbb E}
\newcommand{\Pp}{\mathbb P}

\newcommand{\dist}{\operatorname{dist}}
\newcommand{\cx}{\mathrm{cx}}

\newcommand{\ip}[2]{\left\langle #1,#2\right\rangle}

\newcommand{\clB}[1]{\overline B(0,#1)}
\DeclareMathOperator{\argmax}{arg\,max}

\title{\textbf{A constructive solution to Talagrand's Gaussian convexification problem.}}
\author{Othmane Mazhar\\
\small Institute Louis Bachelier --- Fondation du Risque, Paris, France\\
\small \texttt{othmane.xx90@gmail.com}}
\date{}

\begin{document}
\maketitle

\begin{abstract}
Talagrand asked for a construction of a large convex subset of a bounded Minkowski sum of a large Gaussian set. We first prove a stronger nonsymmetric existential statement: if $A\subset\R^n$ is measurable and $\gamma_n(A)>2/3$, then $A+A+A$ contains a compact convex set of Gaussian measure at least $1/2$. Let now $A$ be closed with $\gamma_n(A)\ge7/8$, let $\Phi$ be the standard Gaussian distribution function, and put $a_0=\Phi^{-1}(3/4)$. For every $\Lambda>1$ and every $0<p<2\Phi(\Lambda a_0)-1$ we construct a centrally symmetric finite polytope $C\subset\Lambda(A+A+A)$ with $\gamma_n(C)\ge p$. Thus measure $3/4$ is obtained for every $\Lambda>1.705510\ldots$. We give matching upper and lower bounds for the optimal high-measure dilation and show that the dilation profile used by the construction is sharp among all symmetric half-measure cores. Finally, if $\gamma_n(A)\ge5/6+\eta$, we construct, for every $0<\varepsilon\le1/2$, a centered ellipsoid $E_\varepsilon$ with $\gamma_n(E_\varepsilon)\ge1-\varepsilon$ and
\begin{equation}
\frac{c}{\Phi^{-1}(1-\varepsilon/2)}\sqrt{\frac{\log n}{n}}\,E_\varepsilon\subset A+A+A.
\end{equation}
Both the dimension dependence and the dependence on $\varepsilon$ are optimal up to constants. The construction also gives a six-summand theorem for balanced sets and a second finite construction based on subgaussian tests.
\end{abstract}



\section{Introduction}

Throughout, $\gamma_n$ denotes the standard Gaussian measure on $\R^n$, and $\Phi$ denotes the standard one-dimensional Gaussian distribution function. The convexifying effect of repeated Minkowski addition is a classical theme in convex geometry. If $A\subset\R^n$ is compact and $A^{(k)}=A+\cdots+A$ denotes its $k$-fold Minkowski sum, then the normalized sums $k^{-1}A^{(k)}$ approach $\operatorname{conv}(A)$ as $k\to\infty$. This phenomenon goes back to the Shapley--Folkman--Starr principle and to the work of Emerson and Greenleaf \cite{Starr1969,EmersonGreenleaf1969}; see \cite{FradeliziMadimanMarsigliettiZvavitch2018} for a systematic account of quantitative convexification under Minkowski addition. These results concern the geometry of the whole sumset and typically measure its distance from the convex hull by a dimension-dependent index of nonconvexity.

Talagrand proposed a different problem in which dimension dependence is forbidden. In \cite{Talagrand1995}, and later in related formulations \cite{Talagrand2010,Talagrand2026}, he asked whether a bounded number of Minkowski additions must already reveal a convex subset which is large for Gaussian measure, with the number of additions and the measure threshold independent of the ambient dimension. The conclusion is deliberately weaker than ordinary convexification: the whole sumset need not become convex, or even close to convex; one asks only for a large convex core inside it. The distinction is essential in high dimension. Gaussian measure has dimension-free concentration and isoperimetric properties, beginning with the Gaussian isoperimetric inequality of Sudakov--Tsirelson and Borell \cite{SudakovTsirelson1974,Borell1975}; see \cite{Ledoux2001} for background. Talagrand's question asks whether Minkowski addition exhibits an analogous dimension-free regularization at the level of convex subsets.

The use of Minkowski addition is itself significant. A bounded Carath\'eodory-type operation does not have the same effect. Johnston recently proved that a natural stronger statement based on convex combinations of a bounded number of points fails in high dimension \cite{Johnston2025}. Thus Talagrand's problem is not simply an approximation-to-the-convex-hull question in another form. It is a problem about the interaction between addition, Gaussian measure and convexity.

This places the problem somewhat orthogonally to the classical Shapley--Folkman picture. In the latter, the ambient dimension controls how many nonconvex summands can contribute to an extreme point of a large sum, and one obtains quantitative closeness to a convex hull after sufficiently many additions. Talagrand asks for no approximation of the whole sum and no small Hausdorff error. The conclusion is instead measure-theoretic: a bounded sum, formed before the dimension is known, must contain one convex region carrying a fixed proportion of Gaussian mass. There is therefore no formal implication from the classical convexification theorems to Talagrand's statement, even though both express a regularizing effect of addition. The distinction also explains why examples in which $A+A+A$ remains disconnected are compatible with the theorem proved here.

A major step was obtained by Song \cite{Song2026}, who connected the geometric problem with representation questions for random vectors and proved several strong consequences, including optimal-scale ellipsoidal estimates. Hua, Song and Tudose subsequently solved Talagrand's Gaussian convexity problem \cite{HuaSongTudose2026}. Talagrand later gave a Hahn--Banach deduction with explicit numerical constants \cite{Talagrand2026}. Their principal structural result is naturally expressed in terms of convex order. Recall that for integrable random vectors $X,Y$ in $\R^n$ one writes $X\preceq_{\cx}Y$ if $\E\varphi(X)\le\E\varphi(Y)$ for every convex function $\varphi$ for which the expectations are defined. Convex order is classical in probability; its relation with martingale couplings goes back to Strassen \cite{Strassen1965}, and general background may be found in \cite{MullerStoyan2002}.

\begin{theorem}[Three-Gaussians theorem, Hua--Song--Tudose]
\label{thm:HST}
Let $G$ be a standard Gaussian random vector in $\R^n$ and let $X$ be an $\R^n$-valued random vector. If $X\preceq_{\cx}G$, then there exist standard Gaussian random vectors $G_1,G_2,G_3$ in $\R^n$, not necessarily independent, such that $X=G_1+G_2+G_3$ almost surely.
\end{theorem}

The number three is geometrically meaningful. Talagrand showed that two summands cannot provide a dimension-free convexification theorem of this kind: for every fixed $L>0$ there are a dimension $N$ and a balanced set $A\subset\R^N$ (that is, $tx\in A$ whenever $x\in A$ and $|t|\le1$) of Gaussian measure at least $3/4$ such that $L(A+A)$ contains no large convex set \cite{Talagrand1995,Talagrand2026}. Thus the occurrence of three Gaussian marginals in \cref{thm:HST} reflects a genuine distinction between double and triple Minkowski sums rather than a peculiarity of the proof.

The first result of this paper is an existential strengthening in which neither symmetry nor an exterior dilation is required. The natural threshold $2/3$ comes directly from the union bound for three Gaussian marginals.

\begin{theorem}[Nonsymmetric three-sum convexification]
\label{thm:nonsymmetric}
Let $A\subset\R^n$ be measurable and suppose that $\gamma_n(A)>2/3$. Then $A+A+A$ contains a compact convex set $C$ with $0\in C$ and $\gamma_n(C)\ge1/2$. Moreover, for every $\lambda>1$, the set $\lambda(A+A+A)$ contains a finite polytope of Gaussian measure at least $1/2$.
\end{theorem}

The proof still uses compactness and is not constructive. After replacing $A$ by a compact subset of measure larger than $2/3$, we place probability measures on a compact part of the complement of the triple sum. No such measure can be dominated by a standard Gaussian in convex order, because \cref{thm:HST} and the union bound would put positive mass back in the triple sum. A finite subcover and a minimax argument combine the resulting convex witnesses into one barrier which is uniformly strict on that compact complement. Kwapie\'n's mean--median comparison then produces a convex half-measure sublevel set. Letting a neighborhood of the triple sum decrease to the triple sum proves the theorem. In particular, this gives a half-measure convex core in the undilated triple sum under the assumption $\gamma_n(A)>2/3$, without any symmetry hypothesis. The argument remains nonconstructive, however, because it does not specify which finite family of convex witnesses must be chosen.

Talagrand asks explicitly \cite[Problem~2.3]{Talagrand2026} for an algorithm which, given a closed set $A\subset\R^n$ with $\gamma_n(A)\ge7/8$, constructs a convex set $C\subset5(A+A+A)$ with $\gamma_n(C)\ge3/4$. Our constructive starting point is that, once finitely many exterior points have been fixed, the separating measure is used only through finitely many evaluations. We therefore separate directly in an evaluation space $\R^m$. The dual variable is then a probability vector rather than an abstract positive functional, hence the probabilistic obstruction is a finitely supported random vector. The three-Gaussians theorem rules out this finite certificate. Supporting hyperplanes produce a polyhedral convex barrier, and Gaussian integrability gives a uniform slope bound on all its affine pieces. This slope bound makes it possible to pass from finitely many prescribed exterior points to the full complement by a finite net and one finite-dimensional convex minimization problem.

There are two further observations which substantially strengthen the numerical conclusion. First, the finite barrier can be chosen even, so its sublevel sets are centrally symmetric. Second, for a convex function of a Gaussian vector a median does not exceed the expectation, by a result of Kwapie\'n \cite{Kwapien1994}. Thus an expectation-one barrier already gives a half-measure unit sublevel, instead of requiring the loss coming from Markov's inequality at a higher level. Once a symmetric convex set of measure $1/2$ has been constructed, the Gaussian $S$-inequality of Lata\l{}a and Oleszkiewicz \cite{LatalaOleszkiewicz1999} gives the sharp universal lower bound for the measure of all its dilates.

Write
\begin{equation}
\label{eq:PsiDef}
\Psi(s)=\gamma_1([-s,s])=2\Phi(s)-1,\qquad a_0=\Psi^{-1}(1/2)=\Phi^{-1}(3/4).
\end{equation}
The constructive result is the following.

\begin{theorem}[Constructive Gaussian convexification]
\label{thm:main}
Let $S\subset\R^n$ be compact and centrally symmetric with $\gamma_n(S)>2/3$. For every $\lambda>1$ one can construct a centrally symmetric finite polytope $K_\lambda\subset\lambda(S+S+S)$ with $\gamma_n(K_\lambda)\ge1/2$. Moreover $S+S+S$ contains a compact centrally symmetric convex set of Gaussian measure at least $1/2$.

Consequently, if $A\subset\R^n$ is closed and $\gamma_n(A)\ge7/8$, then for every $\Lambda>1$ and every $0<p<\Psi(\Lambda a_0)$ one can construct a centrally symmetric finite polytope $C\subset\Lambda(A+A+A)$ with $\gamma_n(C)\ge p$. The construction uses finitely many compact finite-dimensional optimization problems followed by one finite-dimensional convex minimization problem.
\end{theorem}

Taking $p=3/4$ gives the threshold $\Lambda>\Phi^{-1}(7/8)/\Phi^{-1}(3/4)=1.705510543\ldots$, compared with the factor $5$ in Talagrand's constructive question. At dilation $4$, \cref{thm:main} constructs polytopes of every Gaussian measure strictly
below $\Psi(4a_0)=0.993023\ldots$.

The upper bound supplied by \cref{thm:main} has a matching lower bound in its dependence on the target measure. To state this precisely, for $0<\varepsilon\le1/2$ put
\begin{equation}
\label{eq:qepsilon}
q_\varepsilon=\Phi^{-1}(1-\varepsilon/2),\qquad b_0=\Phi^{-1}(15/16),
\end{equation}
For $0<\varepsilon<1/2$, let $\Lambda_*(\varepsilon)$ be the infimum of all $\Lambda\ge1$ such that, for every $n$ and every closed $A\subset\R^n$ with $\gamma_n(A)\ge7/8$, the set $\Lambda(A+A+A)$ contains a convex subset of Gaussian measure at least $1-\varepsilon$.

\begin{theorem}[Optimality of the measure--dilation tradeoff]
\label{thm:optimality}
For every $0<\varepsilon<1/2$,
\begin{equation}
\label{eq:optimalTradeoff}
\max\left\{1,\frac{q_\varepsilon}{3b_0}\right\}
\le \Lambda_*(\varepsilon)
\le \frac{q_\varepsilon}{a_0}.
\end{equation}
For every strict inequality $\Lambda>q_\varepsilon/a_0$, the upper bound is attained by the constructive finite-polytope procedure of \cref{thm:main}. In particular,
\begin{equation}
\label{eq:optimalTradeoffOrder}
\Lambda_*(\varepsilon)=\Theta\bigl(\sqrt{\log(1/\varepsilon)}\bigr)
\qquad(\varepsilon\downarrow0).
\end{equation}
Moreover, for every $n\ge1$ and every $s\ge1$,
\begin{equation}
\label{eq:optimalSprofile}
\inf\left\{\gamma_n(sK):K=-K\text{ closed and convex},\ \gamma_n(K)\ge\frac12\right\}
=\Psi(sa_0).
\end{equation}
The same infimum is obtained if one restricts to compact centrally symmetric convex bodies, or even to centrally symmetric finite polytopes.
\end{theorem}

Thus the high-measure dilation in the constructive theorem has the optimal order, already for one-dimensional convex inputs. The second assertion makes a different and exact optimality statement. Once one retains only the information that the constructed core is symmetric, convex and has measure $1/2$, the profile $\Psi(sa_0)$ and hence the threshold $1.705510543\ldots$ for measure $3/4$ cannot be improved. This does not assert that the same numerical threshold is optimal for Talagrand's problem itself: an improvement there would have to use additional information about the particular cores produced by the barrier construction.

The finite description also gives a stronger ellipsoid conclusion. Song proved that a bounded sum of a large Gaussian set contains an ellipsoid at the optimal scale $\sqrt{(\log n)/n}$ and observed that the ellipsoid in his argument is not explicitly constructed \cite{Song2026}. Combining the finite half-space description with a semidefinite feasibility problem gives not only a constructive ellipsoid, but ellipsoids of arbitrarily high prescribed Gaussian measure inside the undilated triple sum.

\begin{theorem}[Constructive high-measure ellipsoid]
\label{thm:constructiveellipsoid}
There is a universal constant $c>0$ with the following property. Let $n\ge2$, let $0<\eta\le1/6$, and let $A\subset\R^n$ be closed with $\gamma_n(A)\ge5/6+\eta$. For every $0<\varepsilon\le1/2$ one can construct a centered ellipsoid $E_\varepsilon$ such that
\begin{equation}
\label{eq:ellipsoidMain}
\gamma_n(E_\varepsilon)\ge1-\varepsilon,
\qquad
\frac{c}{q_\varepsilon}\sqrt{\frac{\log n}{n}}\,E_\varepsilon
\subset A+A+A.
\end{equation}
The construction first uses finitely many compact finite-dimensional optimization problems and one finite-dimensional convex minimization. It then obtains the ellipsoid from a finite semidefinite feasibility problem. The order $\sqrt{(\log n)/n}$ and the factor $q_\varepsilon^{-1}$ are both optimal up to constants.
\end{theorem}

Talagrand's hypothesis $\gamma_n(A)\ge7/8$ is the special case $\eta=1/24$. The point of \cref{thm:constructiveellipsoid} is that the final ellipsoid lies in $A+A+A$ itself. Its measure can tend to one; only the scale at which it is contained in the triple sum decreases, necessarily, like $q_\varepsilon^{-1}$. For balanced sets the same symmetric construction has no symmetrization loss and yields a convex set of measure $3/4$ in six Minkowski summands. More generally, $3m$ summands give every measure below $\Psi(ma_0)$, and the number of summands required for measure $1-\varepsilon$ has the optimal order $\sqrt{\log(1/\varepsilon)}$.

The paper is organized as follows. The next section proves the nonsymmetric existential theorem. The following section gives the complete finite construction, including the symmetric separation, slope estimate, grid, convex program, the application of Kwapie\'n's mean--median comparison, the $S$-inequality and the endpoint argument. We then prove the two optimality statements and the constructive high-measure ellipsoid theorem. The remaining sections give examples, the balanced case, and an independent construction based on subgaussian tests, before recording the size of the finite construction and the remaining questions.

\section{A nonsymmetric three-sum theorem}

\begin{proof}[Proof of \cref{thm:nonsymmetric}]
By inner regularity, there is a compact set $A_0\subset A$ such that
\begin{equation}
\label{eq:nonsymmetricA0}
\gamma_n(A_0)>\frac23.
\end{equation}
Put $D=A_0+A_0+A_0$. The set $D$ is compact. For $\rho>0$ let
\begin{equation}
\label{eq:nonsymmetricNeighborhood}
D_\rho=\{x\in\R^n:\dist(x,D)<\rho\}.
\end{equation}
Choose $R=R_\rho$ sufficiently large that $D_\rho\subset B(0,R-1)$, and set
\begin{equation}
\label{eq:nonsymmetricComplement}
L_\rho=\clB R\setminus D_\rho.
\end{equation}
This is a compact set. We shall construct a compact convex set of Gaussian measure at least $1/2$ inside $D_\rho$, with a bound independent of $\rho$, and then let $\rho$ decrease to zero.

Let $\mathcal P(L_\rho)$ be the space of Borel probability measures on $L_\rho$, equipped with the weak topology. It is compact and convex. We first claim that for every $\mu\in\mathcal P(L_\rho)$ there is a finite-valued convex function $f_\mu:\R^n\to\R$, integrable with respect to Gaussian measure, such that
\begin{equation}
\label{eq:nonsymmetricWitness}
\int f_\mu\,d\mu>\E f_\mu(G).
\end{equation}
Indeed, let $X$ have law $\mu$ and suppose to the contrary that $X\preceq_{\cx}G$. By \cref{thm:HST}, there are standard Gaussian random vectors $G_1,G_2,G_3$ such that $X=G_1+G_2+G_3$ almost surely. Using only their marginal laws and \eqref{eq:nonsymmetricA0},
\begin{equation}
\label{eq:nonsymmetricUnionBound}
\Pp(G_1,G_2,G_3\in A_0)
\ge 3\gamma_n(A_0)-2>0.
\end{equation}
On this event $X\in D$. This is impossible because $X$ is supported on $L_\rho$, which is disjoint from $D$. Hence $X\not\preceq_{\cx}G$, and the definition of convex order gives a convex witness satisfying \eqref{eq:nonsymmetricWitness}. The strict inequality itself ensures that the Gaussian expectation of the witness is finite.

For every finite convex function $f$ with $\E f(G)<\infty$, put
\begin{equation}
\label{eq:nonsymmetricOpenSet}
\mathcal U_f=\left\{\mu\in\mathcal P(L_\rho):
\int f\,d\mu-\E f(G)>0\right\}.
\end{equation}
Since a finite convex function is continuous and $L_\rho$ is compact, the map $\mu\mapsto\int f\,d\mu$ is continuous on $\mathcal P(L_\rho)$. The sets $\mathcal U_f$ are therefore open. By the preceding paragraph they cover $\mathcal P(L_\rho)$, so compactness supplies a finite subcover associated with convex functions $f_1,\ldots,f_m$.

For $1\le j\le m$ define
\begin{equation}
\label{eq:nonsymmetricLinearFunctionals}
L_j(\mu)=\int f_j\,d\mu-\E f_j(G).
\end{equation}
The continuous function $\mu\mapsto\max_jL_j(\mu)$ is strictly positive on the compact space $\mathcal P(L_\rho)$. Consequently
\begin{equation}
\label{eq:nonsymmetricDelta}
\delta_\rho:=\min_{\mu\in\mathcal P(L_\rho)}\max_{1\le j\le m}L_j(\mu)>0.
\end{equation}
Let $\Delta_m$ denote the probability simplex. Since the expression $\sum_jp_jL_j(\mu)$ is affine and continuous in both variables, the minimax theorem \cite{Sion1958} gives
\begin{align}
\delta_\rho
&=\min_{\mu\in\mathcal P(L_\rho)}
  \max_{p\in\Delta_m}\sum_{j=1}^mp_jL_j(\mu)\notag\\
&=\max_{p\in\Delta_m}
  \min_{\mu\in\mathcal P(L_\rho)}\sum_{j=1}^mp_jL_j(\mu).
\label{eq:nonsymmetricMinimax}
\end{align}
Choose $p=(p_1,\ldots,p_m)\in\Delta_m$ attaining the last maximum and set
\begin{equation}
\label{eq:nonsymmetricCombinedBarrier}
f_\rho=\sum_{j=1}^mp_jf_j,
\qquad
m_\rho=\E f_\rho(G).
\end{equation}
Taking $\mu$ to be the Dirac mass at $x\in L_\rho$ in \eqref{eq:nonsymmetricMinimax} gives
\begin{equation}
\label{eq:nonsymmetricStrictBarrier}
f_\rho(x)\ge m_\rho+\delta_\rho
\qquad(x\in L_\rho).
\end{equation}

Jensen's inequality gives $f_\rho(0)\le m_\rho$. Define
\begin{equation}
\label{eq:nonsymmetricCoreRho}
C_\rho=\{x\in\R^n:f_\rho(x)\le m_\rho\}.
\end{equation}
This set is closed and convex and contains the origin. By Kwapie\'n's mean--median comparison for convex functions of Gaussian vectors \cite{Kwapien1994}, a median of $f_\rho(G)$ does not exceed $m_\rho=\E f_\rho(G)$. Hence
\begin{equation}
\label{eq:nonsymmetricHalfMeasure}
\gamma_n(C_\rho)\ge\frac12.
\end{equation}
We next check that $C_\rho\subset D_\rho$. If $x\in C_\rho\cap\clB R$ and $x\notin D_\rho$, then $x\in L_\rho$, contradicting \eqref{eq:nonsymmetricStrictBarrier}. If instead $x\in C_\rho$ and $|x|>R$, put $y=Rx/|x|$. The choice of $R$ gives $y\in L_\rho$. With $\theta=R/|x|\in(0,1)$, convexity and Jensen's inequality give
\begin{equation}
f_\rho(y)\le\theta f_\rho(x)+(1-\theta)f_\rho(0)\le m_\rho,
\end{equation}
again contradicting \eqref{eq:nonsymmetricStrictBarrier}. Thus
\begin{equation}
\label{eq:nonsymmetricContainmentRho}
C_\rho\subset D_\rho.
\end{equation}
In particular $C_\rho$ is compact.

Let $\rho_k\downarrow0$. All the sets $C_{\rho_k}$ lie, for large $k$, in one fixed compact ball. By Blaschke's selection theorem \cite{Schneider2014}, a subsequence converges in Hausdorff distance to a nonempty compact convex set $C$. Since $0\in C_{\rho_k}$, one has $0\in C$. If $x\in C$, choose $x_k\in C_{\rho_k}$ with $x_k\to x$. From \eqref{eq:nonsymmetricContainmentRho}, $\dist(x_k,D)<\rho_k$, and the closedness of $D$ gives $x\in D$. Hence
\begin{equation}
\label{eq:nonsymmetricLimitContainment}
C\subset D\subset A+A+A.
\end{equation}
For every $t>0$, Hausdorff convergence gives $C_{\rho_k}\subset C+tB_2^n$ for all sufficiently large $k$. Therefore
\begin{equation}
\frac12\le\limsup_k\gamma_n(C_{\rho_k})
\le\gamma_n(C+tB_2^n).
\end{equation}
As $t\downarrow0$, the compact sets $C+tB_2^n$ decrease to $C$, and continuity from above yields $\gamma_n(C)\ge1/2$. This proves the first assertion.

It remains to obtain a finite polytope after an arbitrary exterior dilation. Since $C$ has positive Gaussian measure, it has nonempty interior. In fact $0\in\operatorname{int}C$. Otherwise a supporting hyperplane at the origin would place $C$ in a closed half-space through the origin. That half-space has Gaussian measure $1/2$, while the compact set $C$ omits an open subset of it, and hence would have Gaussian measure strictly smaller than $1/2$, a contradiction.

It follows that $C\subset\operatorname{int}(\lambda C)$ whenever $\lambda>1$. Indeed, choose $r>0$ with $rB_2^n\subset C$. If $x\in C$ and $|z|<r$, then
\begin{equation}
x+(\lambda-1)z
=\lambda\left(\frac1\lambda x+\left(1-\frac1\lambda\right)z\right)
\in\lambda C.
\end{equation}
Moreover $C$ has positive Lebesgue measure, so $\lambda C\setminus C$ has positive Lebesgue, and hence positive Gaussian, measure. Consequently $\gamma_n(\lambda C)>\gamma_n(C)\ge1/2$. A convex body is the Hausdorff limit from within of finite polytopes \cite{Schneider2014}. Since the boundary of a convex body has Gaussian measure zero, there are finite polytopes $P_j\subset\lambda C$ with $\gamma_n(P_j)\to\gamma_n(\lambda C)$. Choose $j$ large enough that $\gamma_n(P_j)\ge1/2$ and write $P_\lambda=P_j$. Finally,
\begin{equation}
P_\lambda\subset\lambda C\subset\lambda D
\subset\lambda(A+A+A),
\end{equation}
which proves the second assertion.
\end{proof}

\section{Finite barriers and proof of the main theorem}

\noindent\textit{1. The compact symmetric set.} Throughout this section let $G\sim N(0,I_n)$. We first explain how the set in Talagrand's question supplies the compact symmetric set in \cref{thm:main}. Set $r=\sqrt{24n}$ and
\begin{equation}
\label{eq:defS}
S=A\cap(-A)\cap\clB r.
\end{equation}
Gaussian invariance under reflection gives $\gamma_n(A\cap(-A))\ge2\gamma_n(A)-1\ge3/4$. Since $\E|G|^2=n$, Markov's inequality gives $\Pp(|G|>r)\le1/24$, and therefore
\begin{equation}
\label{eq:Smeasure}
\gamma_n(S)\ge\frac34-\frac1{24}=\frac{17}{24}>\frac23.
\end{equation}
The set $S$ is compact because $A$ is closed, and it is centrally symmetric by construction. More generally, the argument below starts from any compact centrally symmetric $S$ with $\gamma_n(S)>2/3$; choose $r$ so that $S\subset\clB r$. Fix $\lambda>1$ and define
\begin{equation}
\label{eq:defD}
D=\lambda(S+S+S).
\end{equation}
Then $D$ is compact and centrally symmetric, and $D\subset\clB{3\lambda r}$. For the set obtained from \eqref{eq:defS}, one also has $D\subset\lambda(A+A+A)$. We shall construct a centrally symmetric polytope of Gaussian measure at least $1/2$ inside $D$.

The normalization of the barriers below requires $0\in D$. This is not automatic from symmetry alone, since $S$ need not contain the origin. Consider the deterministic random vector $X=0$. For every finite convex function $\varphi$, Jensen's inequality gives $\varphi(0)=\varphi(\E G)\le\E\varphi(G)$, and hence $0\preceq_{\cx}G$. By \cref{thm:HST}, there exist standard Gaussian vectors $G_1,G_2,G_3$ such that $G_1+G_2+G_3=0$ almost surely. Since each marginal has law $\gamma_n$, the union bound gives $\Pp(G_1,G_2,G_3\in S)\ge3\gamma_n(S)-2>0$. On this event $0=G_1+G_2+G_3\in S+S+S$, and therefore $0\in D$.

\medskip\noindent\textit{2. Finite symmetric separation.} We now replace the infinite-dimensional Hahn--Banach step by separation in a finite evaluation space. We enforce the symmetry of $D$ already at the level of the test functions. This gives a symmetric dual certificate and, ultimately, a barrier whose sublevel sets are centrally symmetric.

\begin{lemma}[Finite symmetric barrier]
\label{lem:finitebarrier}
Let $x_1,\ldots,x_m\in D^c$. Then there are vectors $u_i\in\R^n$ and numbers $v_i\le0$ such that
\begin{equation}
\label{eq:qFinite}
q(x)=\max\bigl(0,|\ip{u_1}{x}|+v_1,\ldots,|\ip{u_m}{x}|+v_m\bigr)
\end{equation}
is even and convex, satisfies $q(x_i)>\lambda$ for every $i$, and $\E q(G)\le1$.
\end{lemma}

\begin{proof}
Let $\mathcal F$ be the class of all finite-valued functions $f:\R^n\to[0,\infty)$ which are convex and even, satisfy $f(0)=0$, and obey $\E f(G)\le1$. We first show that there exists $f\in\mathcal F$ such that $f(x_i)>\lambda$ for every $i$. Suppose, to the contrary, that no such function exists. Put $V=\{(f(x_1),\ldots,f(x_m)):f\in\mathcal F\}\subset\R^m$ and $U=(\lambda,\infty)^m$. The set $V$ is convex because $\mathcal F$ is convex and the evaluation map is linear, while $U$ is open and convex. Our assumption is exactly $V\cap U=\varnothing$. By finite-dimensional convex separation there exist a nonzero vector $p=(p_1,\ldots,p_m)\in\R^m$ and a real number $\alpha$ such that
\begin{equation}
\label{eq:finiteSeparation}
\sup_{v\in V}\ip{p}{v}\le\alpha\le\inf_{u\in U}\ip{p}{u}.
\end{equation}
Necessarily $p_i\ge0$ for every $i$. Indeed, if $p_i<0$, then letting the $i$th coordinate of $u\in U$ tend to $+\infty$ would force $\inf_{u\in U}\ip{p}{u}=-\infty$, contradicting \eqref{eq:finiteSeparation}. Multiplying $p$ by a positive scalar, we may assume $\sum_i p_i=1$. The right-hand infimum in \eqref{eq:finiteSeparation} is then $\lambda$, and therefore
\begin{equation}
\label{eq:separationExpectation}
\sum_{i=1}^m p_i f(x_i)\le\lambda\qquad(f\in\mathcal F).
\end{equation}

Let $Y$ be the finitely supported random vector defined by $\Pp(Y=x_i)=p_i$. Let $\varepsilon$ be an independent symmetric sign and put $X=\varepsilon Y$. Since $D=-D$ and every $x_i\notin D$, both $x_i$ and $-x_i$ belong to $D^c$, hence
\begin{equation}
\label{eq:XoutsideD}
\Pp(X\in D)=0.
\end{equation}
We claim that $X/\lambda\preceq_{\cx}G$. The point of the even test class is that it still controls arbitrary convex functions once the random vector has been symmetrized.

Let first $h:\R^n\to[0,\infty)$ be finite-valued and convex, with $h(0)=0$ and $\E h(G)<\infty$. Define the even part $h_e(x)=(h(x)+h(-x))/2$. Then $h_e$ is even, nonnegative and convex, $h_e(0)=0$, and, because $G$ is symmetric, $\E h_e(G)=\E h(G)$. Since $X=\varepsilon Y$ with an independent symmetric sign, $\E h(X)=\frac12\E h(Y)+\frac12\E h(-Y)=\E h_e(Y)$. If $\E h(G)>0$, then $h_e/\E h(G)\in\mathcal F$, and \eqref{eq:separationExpectation} yields $\E h(X)\le\lambda\E h(G)$. If $\E h(G)=0$, then $h(G)=0$ almost surely. A finite convex function is continuous on $\R^n$, and Gaussian measure has full support, so $h\equiv0$ and the same inequality is again valid. Since $\lambda>1$, we may write $X/\lambda=\lambda^{-1}X+(1-\lambda^{-1})0$; convexity and $h(0)=0$ therefore give $h(X/\lambda)\le\lambda^{-1}h(X)$. Consequently
\begin{equation}
\label{eq:hConvexOrder}
\E h(X/\lambda)\le\E h(G).
\end{equation}

Now let $\varphi:\R^n\to\R$ be an arbitrary finite-valued convex function. Since its domain is all of $\R^n$, the subdifferential $\partial\varphi(0)$ is nonempty; choose $a\in\partial\varphi(0)$ and set $h(x)=\varphi(x)-\varphi(0)-\ip{a}{x}$. The supporting-hyperplane inequality gives $h\ge0$ and $h(0)=0$. The random vector $X$ is symmetric and hence centered, while $\E G=0$. Thus the affine part has the same expectation under $X/\lambda$ and $G$, and \eqref{eq:hConvexOrder} gives $\E\varphi(X/\lambda)\le\E\varphi(G)$ whenever the latter expectation is finite. If $\E\varphi(G)=+\infty$, there is nothing to prove. Hence $X/\lambda\preceq_{\cx}G$.

By the three-Gaussians theorem there exist standard Gaussian vectors $G_1,G_2,G_3$, not necessarily independent, such that
\begin{equation}
\label{eq:XthreeGaussians}
X/\lambda=G_1+G_2+G_3\qquad\text{almost surely}.
\end{equation}
Using only the marginal laws, $\Pp(G_1,G_2,G_3\in S)\ge3\gamma_n(S)-2>0$. On this event $X=\lambda(G_1+G_2+G_3)\in\lambda(S+S+S)=D$, contradicting \eqref{eq:XoutsideD}. This proves the existence of $f\in\mathcal F$ with $f(x_i)>\lambda$ for all $i$.

For each $i$, choose a subgradient $u_i\in\partial f(x_i)$ and let $\ell_i(x)=f(x_i)+\ip{u_i}{x-x_i}=\ip{u_i}{x}+v_i$. The supporting-hyperplane inequality gives $\ell_i(x)\le f(x)$ for all $x$, and evaluating at $0$ gives $v_i=\ell_i(0)\le f(0)=0$. Moreover $\ell_i(x_i)=f(x_i)>\lambda$. Since $f$ is even, $\ell_i(-x)\le f(-x)=f(x)$ for every $x$. Hence both affine functions $x\mapsto\ip{u_i}{x}+v_i$ and $x\mapsto-\ip{u_i}{x}+v_i$ lie below $f$. Their maximum is $|\ip{u_i}{x}|+v_i$, and therefore the function in \eqref{eq:qFinite} satisfies $0\le q\le f$. It follows that $\E q(G)\le1$, while $q(x_i)\ge\ell_i(x_i)>\lambda$. The barrier is even by construction.
\end{proof}

The proof above is exactly the finite-dimensional analogue of the positive-functional step in the existential argument: the exterior points determine the finite evaluation map, the separating functional is represented by the probability vector $p$, and the resulting probabilistic certificate has finite support. The reason for imposing evenness before performing the separation is also worth emphasizing. If one first constructed an arbitrary barrier and then replaced it by $(q(x)+q(-x))/2$, the value at a prescribed point $x_i$ could fall below the required separation height, because the original construction need not control $-x_i$. Here the symmetry is already built into the test class and into the dual random vector $X=\varepsilon Y$, so each pair $\{x_i,-x_i\}$ is controlled simultaneously. This is what later allows the unit sublevel of the barrier to be fed directly into the Gaussian $S$-inequality.

\medskip\noindent\textit{3. The slope bound and the finite net.} Set $R=3\lambda r+1$. Since $D\subset\clB{3\lambda r}$, one has $D\subset B(0,R-1)$. Let $Z\sim N(0,1)$, define $\beta_R=\E(Z-R)_+$, and put $L=\beta_R^{-1}$. Since a standard Gaussian has positive mass on $(R,\infty)$, $\beta_R>0$.

\begin{lemma}[Uniform slope bound]
\label{lem:slope}
Suppose $x_1,\ldots,x_m\in\clB R$ and $q$ has the form \eqref{eq:qFinite}, with $v_i\le0$, $\ip{u_i}{x_i}+v_i\ge\lambda$ for every $i$, and $\E q(G)\le1$. Then $|u_i|\le L$ for every $i$, and consequently $q$ is globally $L$-Lipschitz.
\end{lemma}

\begin{proof}
Fix $i$. The constraint $\ip{u_i}{x_i}+v_i\ge\lambda$ and $v_i\le0$ imply $u_i\ne0$. Write $\sigma_i=|u_i|$, $e_i=u_i/|u_i|$, and $\tau_i=-v_i/\sigma_i$. Then $\tau_i\ge0$, while the constraint at $x_i$ becomes $\sigma_i(\ip{e_i}{x_i}-\tau_i)\ge\lambda$. In particular $\ip{e_i}{x_i}-\tau_i>0$, so $0\le\tau_i<\ip{e_i}{x_i}\le|x_i|\le R$.

The barrier $q$ dominates the positive part of each supporting affine function, and therefore $q(G)\ge(\ip{u_i}{G}+v_i)_+$. Since $\ip{e_i}{G}$ is a standard one-dimensional Gaussian,
\begin{equation}
\label{eq:slopeEstimate}
1\ge\E q(G)\ge\E(\ip{u_i}{G}+v_i)_+=\sigma_i\E(Z-\tau_i)_+\ge\sigma_i\E(Z-R)_+=\sigma_i\beta_R.
\end{equation}
Thus $\sigma_i\le\beta_R^{-1}=L$. The map $x\mapsto|\ip{u_i}{x}|+v_i$ is $|u_i|$-Lipschitz, hence $L$-Lipschitz. The constant function $0$ is also $L$-Lipschitz, and a maximum of finitely many $L$-Lipschitz functions is again $L$-Lipschitz. This proves the claim.
\end{proof}

This estimate is the quantitative point which turns a finite certificate into a global construction. The finite barrier lemma controls only the points at which we impose constraints. The Gaussian expectation bound prevents the corresponding supporting hyperplanes from acquiring arbitrarily large slopes on the fixed ball which contains $D$, so a Euclidean net with a prescribed mesh controls all of the complement at once.

The final half-measure core will be the unit sublevel of the barrier. We use half of the gap between the separation height $\lambda$ and the final level $1$ to pass from the net to the full complement. Set
\begin{equation}
\label{eq:Ldelta}
\delta=\frac{\lambda-1}{2L}=\frac{\lambda-1}{2}\beta_R,
\end{equation}
and put $h=\delta/\sqrt n$. Consider the standard axis-parallel grid of closed cubes $Q_k=hk+[0,h]^n$, where $k\in\mathbb Z^n$. Every such cube has Euclidean diameter exactly $\delta$, and only finitely many grid cubes meet $\clB R$; enumerate them as $Q_1,\ldots,Q_N$.

For each $j$ define
\begin{equation}
\label{eq:rhoj}
\rho_j=\max_{x\in Q_j\cap\clB R}\dist(x,D).
\end{equation}
The maximum exists because $Q_j\cap\clB R$ is compact and the distance function to the closed set $D$ is continuous. Whenever $\rho_j>0$, choose $x_j\in\argmax_{x\in Q_j\cap\clB R}\dist(x,D)$. Then $\dist(x_j,D)=\rho_j>0$, so $x_j\notin D$. Discard the cubes for which $\rho_j=0$ and relabel the retained points as $x_1,\ldots,x_m$.

We verify explicitly that these points form a $\delta$-net of $\clB R\setminus D$. Take $x\in\clB R\setminus D$. Because $D$ is closed, $\dist(x,D)>0$. Choose any grid cube $Q_j$ containing $x$. Then $\rho_j\ge\dist(x,D)>0$, so this cube was retained. Its chosen point $x_j$ lies in the same cube as $x$, and the diameter of the cube is $\delta$; hence $|x-x_j|\le\delta$. This proves the covering property without requiring any regularity of the boundary of $D$.

The compact optimization defining each retained point depends explicitly on the original set. Since $D=\lambda(S+S+S)$,
\begin{equation}
\label{eq:distanceD}
\dist(x,D)=\min_{s_1,s_2,s_3\in S}|x-\lambda(s_1+s_2+s_3)|.
\end{equation}
The minimum is attained because $S^3$ is compact. A crude count of the cubes meeting $\clB R$ gives
\begin{equation}
\label{eq:facetbound}
m\le\left(2+\frac{2R\sqrt n}{\delta}\right)^n=\left(2+\frac{4R\sqrt n}{(\lambda-1)\beta_R}\right)^n.
\end{equation}
We will use this only to make the finiteness of the construction explicit; no attempt is made to optimize it.

\medskip\noindent\textit{4. The convex program and the global barrier.} Apply \cref{lem:finitebarrier} to the net points $x_1,\ldots,x_m$. This suggests the following finite convex optimization problem:
\begin{equation}
\label{eq:convexProgram}
\boxed{\begin{aligned}
\text{minimize}\quad&\Psi_0(u,v):=\E\max\bigl(0,|\ip{u_1}{G}|+v_1,\ldots,|\ip{u_m}{G}|+v_m\bigr),\\
\text{subject to}\quad&v_i\le0,\qquad i=1,\ldots,m,\\
&\ip{u_i}{x_i}+v_i\ge\lambda,\qquad i=1,\ldots,m.
\end{aligned}}
\end{equation}
The variables are $(u_i,v_i)\in\R^n\times\R$, $i=1,\ldots,m$, so the problem has $m(n+1)$ real variables. For every fixed $x$, the map $(u_i,v_i)\mapsto|\ip{u_i}{x}|+v_i$ is convex. A maximum of convex functions is convex, and expectation preserves convexity, so the objective in \eqref{eq:convexProgram} is convex in all the coefficients.

The minimum is attained and is at most $1$. Indeed, \cref{lem:finitebarrier} gives a feasible family of coefficients for which the objective is at most $1$, so the feasible sublevel $\mathcal K=\{(u,v):(u,v)\text{ is feasible and }\Psi_0(u,v)\le1\}$ is nonempty. We show that $\mathcal K$ is compact. First, by \cref{lem:slope}, every point of $\mathcal K$ satisfies $|u_i|\le L$. Since $|x_i|\le R$, the constraint at $x_i$ gives $v_i\ge\lambda-\ip{u_i}{x_i}\ge\lambda-|u_i||x_i|\ge\lambda-LR$, while feasibility gives $v_i\le0$. Thus all coefficients lie in a fixed bounded set.

It remains to check closedness. The linear constraints are closed. Let $(u^{(k)},v^{(k)})\to(u,v)$ in the bounded coefficient region just described. For every fixed $x$, the corresponding maxima converge pointwise. Moreover $|u_i^{(k)}|\le L$ and $\lambda-LR\le v_i^{(k)}\le0$, so the integrands are bounded by $C(1+|G|)$ for a finite constant $C=C(R,L,\lambda)$ independent of $k$. This dominating function is integrable under Gaussian measure. Dominated convergence therefore gives $\Psi_0(u^{(k)},v^{(k)})\to\Psi_0(u,v)$. Hence $\mathcal K$ is closed and bounded in a finite-dimensional space, therefore compact. Since no feasible point with objective larger than $1$ can improve on the feasible point supplied by \cref{lem:finitebarrier}, minimizing over $\mathcal K$ yields a global minimizer of \eqref{eq:convexProgram}.

Choose one such minimizer and set
\begin{equation}
\label{eq:finalq}
q(x)=\max\bigl(0,|\ip{u_1}{x}|+v_1,\ldots,|\ip{u_m}{x}|+v_m\bigr).
\end{equation}
Then $q$ is even and convex, $q(0)=0$, $q(x_i)\ge\lambda$ for every $i$, and $\E q(G)\le1$. By \cref{lem:slope}, $q$ is globally $L$-Lipschitz.

We now pass from the finite net to all of $D^c$. First let $x\in\clB R\setminus D$. Choose a net point $x_i$ with $|x-x_i|\le\delta$. Using the Lipschitz bound and the constraint at $x_i$,
\begin{equation}
\label{eq:barrierInsideBall}
q(x)\ge q(x_i)-L|x-x_i|\ge\lambda-L\delta=\frac{\lambda+1}{2}>1.
\end{equation}
It remains to treat $|x|>R$. Put $y=Rx/|x|$. Then $|y|=R$, and since $D\subset B(0,R-1)$ one has $y\notin D$. The preceding estimate gives $q(y)\ge(\lambda+1)/2$. Write $y=\theta x+(1-\theta)0$ with $\theta=R/|x|\in(0,1)$. By convexity of $q$ and $q(0)=0$, $q(y)\le\theta q(x)$, and therefore $q(x)\ge q(y)/\theta\ge(\lambda+1)/2$. Thus
\begin{equation}
\label{eq:globalBarrier}
x\notin D\quad\Longrightarrow\quad q(x)>1.
\end{equation}
Consequently the unit sublevel
\begin{equation}
\label{eq:Klambda}
K_\lambda=\{x\in\R^n:q(x)\le1\}=\bigcap_{i=1}^m\{x:|\ip{u_i}{x}|\le1-v_i\}
\end{equation}
is a centrally symmetric finite polyhedron contained in $D$. Since $D$ is compact, $K_\lambda$ is bounded and hence is in fact a centrally symmetric polytope. The representation in \eqref{eq:Klambda} uses at most $2m$ half-spaces.

\medskip\noindent\textit{5. Gaussian measure and dilation.}
By Kwapie\'n's mean--median comparison \cite{Kwapien1994}, a median of
$q(G)$ does not exceed $\E q(G)\le1$. Therefore
\begin{equation}
\label{eq:halfmeasure}
\gamma_n(K_\lambda)=\Pp(q(G)\le1)\ge\frac12.
\end{equation}

We now use the $S$-inequality of Lata\l{}a and Oleszkiewicz \cite{LatalaOleszkiewicz1999}. If $K=-K$ is a closed convex set in $\R^n$, $s\ge1$, and $a=\Psi^{-1}(\gamma_n(K))$, then
\begin{equation}
\label{eq:Sineq}
\gamma_n(sK)\ge\Psi(sa).
\end{equation}
Equivalently, among centrally symmetric convex sets with a prescribed Gaussian measure, a symmetric strip has the slowest growth under dilation. Since \eqref{eq:halfmeasure} gives $\Psi^{-1}(\gamma_n(K_\lambda))\ge a_0$, it follows that
\begin{equation}
\label{eq:dilationMeasure}
\gamma_n(sK_\lambda)\ge\Psi(sa_0)\qquad(s\ge1).
\end{equation}

\begin{proof}[Proof of the constructive assertions in \cref{thm:main}]
The construction above proves the first assertion: for every compact centrally symmetric $S$ with $\gamma_n(S)>2/3$ and every $\lambda>1$, it gives $K_\lambda\subset\lambda(S+S+S)$ satisfying \eqref{eq:halfmeasure}. Now let $A$ be as in the second assertion and take the set $S$ in \eqref{eq:defS}. Fix $\Lambda>1$ and $0<p<\Psi(\Lambda a_0)$. By continuity and strict monotonicity of $\Psi$, one can choose $1<\lambda<\Lambda$ so close to $1$ that $p<\Psi((\Lambda/\lambda)a_0)$. Perform the construction with this value of $\lambda$, put $s=\Lambda/\lambda$ and define $C=sK_\lambda$. Scaling the right-hand sides in \eqref{eq:Klambda} gives the half-space representation of $C$, while $C\subset s\lambda(A+A+A)=\Lambda(A+A+A)$ and \eqref{eq:dilationMeasure} gives $\gamma_n(C)\ge\Psi(sa_0)>p$.
\end{proof}
The improvement in the numerical constant occurs entirely in the final Gaussian measure estimate. If one uses only Markov's inequality after constructing a nonnegative barrier $q$ with $\E q(G)\le1$, then
\begin{equation}
\gamma_n\{q\le t\}\ge1-\frac1t,
\end{equation}
so measure $3/4$ requires $t\ge4$. By contrast, Kwapie\'n's mean--median comparison gives a half-measure sublevel already at level
$1$, and the $S$-inequality then yields the full dilation profile \eqref{eq:dilationMeasure}. Thus the improvement comes from preserving symmetry and changing the final measure estimate, not from a finer grid or a better Lipschitz bound.

The input threshold can also be weakened. Suppose $\gamma_n(A)\ge5/6+\eta$ for some $\eta>0$. Then $\gamma_n(A\cap(-A))\ge2/3+2\eta$. Set $r_\eta=\sqrt{n/\eta}$ and $S_\eta=A\cap(-A)\cap\clB{r_\eta}$. Markov's inequality gives $\Pp(|G|>r_\eta)\le\eta$, hence $\gamma_n(S_\eta)\ge2/3+\eta>2/3$. Inspection of the proof shows that this strict inequality is the only measure property of the compact symmetric core used in the finite separation. The entire argument therefore remains valid. We keep the assumption $7/8$ in \cref{thm:main} because it is the hypothesis of Talagrand's constructive problem and gives the convenient radius $\sqrt{24n}$.

\medskip\noindent\textit{6. The endpoint and further consequences of symmetry.} The finite construction uses the positive gap $\lambda-1$, but at the existential level it yields an undilated half-measure core. Choose $\lambda_j\downarrow1$ and let $K_j=K_{\lambda_j}$. For large $j$, say $\lambda_j\le2$, all $K_j$ lie in the fixed compact set $2(S+S+S)$. Each $K_j$ is nonempty because it contains the origin. By Blaschke's selection theorem \cite{Schneider2014}, after passing to a subsequence the sets $K_j$ converge in Hausdorff distance to a nonempty compact centrally symmetric convex set $K$.

We first check the containment. If $x\in K$, choose $x_j\in K_j$ with $x_j\to x$. Write $x_j=\lambda_j(s_{j1}+s_{j2}+s_{j3})$ with $s_{jk}\in S$. Since $S^3$ is compact, a subsequence of the triples converges to $(s_1,s_2,s_3)\in S^3$, and then $x=s_1+s_2+s_3$. Hence $K\subset S+S+S$; for the set in \eqref{eq:defS}, this also gives $K\subset A+A+A$. To pass the measure to the limit, fix $\varepsilon>0$. Hausdorff convergence gives $K_j\subset K+\varepsilon B_2^n$ for all sufficiently large $j$, and therefore
\begin{equation}
\frac12\le\limsup_j\gamma_n(K_j)\le\gamma_n(K+\varepsilon B_2^n).
\end{equation}
As $\varepsilon\downarrow0$, the compact sets $K+\varepsilon B_2^n$ decrease to $K$, so continuity from above gives $\gamma_n(K)\ge1/2$. This proves the endpoint assertion in \cref{thm:main}. Applying the $S$-inequality to $K$ yields $\gamma_n(sK)\ge\Psi(sa_0)$ for every $s\ge1$. Thus the strict inequality in the constructive statement is a feature of the finite construction, not of the existential measure--dilation curve.

The symmetry of the output also controls all central linear sections simultaneously. Let $C=-C$ be any convex set and let $H\subset\R^n$ be a linear subspace. Write $\R^n=H\oplus H^\perp$ and define $F(z)=\int_H\mathbf 1_C(h+z)\,d\gamma_H(h)$ for $z\in H^\perp$. The function $(h,z)\mapsto\mathbf 1_C(h+z)e^{-|h|^2/2}$ is log-concave in the extended sense. By Pr\'ekopa's theorem \cite{Prekopa1973}, $F$ is log-concave on $H^\perp$. Since $C=-C$, the function $F$ is even. An even log-concave function attains its maximum at the origin, so $F(z)\le F(0)$ for all $z$. Fubini's theorem therefore gives
\begin{equation}
\label{eq:sectionInequality}
\gamma_n(C)=\int_{H^\perp}F(z)\,d\gamma_{H^\perp}(z)\le F(0)=\gamma_H(C\cap H).
\end{equation}
Thus a single polytope constructed in \cref{thm:main} has Gaussian measure at least $p$ in every central linear subspace simultaneously. This statement requires no additional optimization after $C$ has been produced.

Finally, because $K_\lambda$ is a bounded finite polyhedron, it is a polytope and may be converted from its half-space representation to a finite vertex representation. If it is described by at most $2m$ half-spaces, a crude combinatorial bound gives at most $\binom{2m}{n}$ vertices. Each vertex $v$ lies in $\lambda(S+S+S)$, and the compact minimization problem $\min_{s_1,s_2,s_3\in S}|v-\lambda(s_1+s_2+s_3)|$ has value zero. Thus the output can, if desired, be supplemented by explicit triples certifying the membership of every vertex in the required triple sum. Since the triple sum need not be convex, these vertex witnesses do not by themselves certify the containment of the whole polytope; that containment is supplied by the global barrier \eqref{eq:globalBarrier}.

\section{Optimality of the measure--dilation tradeoff}

\begin{proof}[Proof of \cref{thm:optimality}]
We begin with the upper bound in \eqref{eq:optimalTradeoff}. Let $A\subset\R^n$ be closed with $\gamma_n(A)\ge7/8$. The endpoint assertion in \cref{thm:main} gives a compact centrally symmetric convex set $K\subset A+A+A$ with $\gamma_n(K)\ge1/2$. Since $q_\varepsilon/a_0\ge1$, the $S$-inequality \eqref{eq:Sineq} gives
\begin{equation}
\label{eq:optimalityUpperEndpoint}
\gamma_n\left(\frac{q_\varepsilon}{a_0}K\right)
\ge\Psi(q_\varepsilon)=1-\varepsilon.
\end{equation}
Moreover $(q_\varepsilon/a_0)K\subset(q_\varepsilon/a_0)(A+A+A)$. This proves
\begin{equation}
\Lambda_*(\varepsilon)\le\frac{q_\varepsilon}{a_0}.
\end{equation}
The endpoint core is obtained by compactness. If instead $\Lambda>q_\varepsilon/a_0$, then
\begin{equation}
1-\varepsilon=\Psi(q_\varepsilon)<\Psi(\Lambda a_0),
\end{equation}
and the constructive assertion in \cref{thm:main}, applied with $p=1-\varepsilon$, produces a centrally symmetric finite polytope $C\subset\Lambda(A+A+A)$ satisfying $\gamma_n(C)\ge1-\varepsilon$.

For the lower bound, work in dimension one and take
\begin{equation}
\label{eq:optimalityInterval}
A=[-b_0,b_0],\qquad b_0=\Phi^{-1}(15/16).
\end{equation}
Then $\gamma_1(A)=\Psi(b_0)=7/8$ and
\begin{equation}
A+A+A=[-3b_0,3b_0].
\end{equation}
Every subset of $\Lambda(A+A+A)$ has Gaussian measure at most $\Psi(3b_0\Lambda)$. Hence the existence of a subset of measure at least $1-\varepsilon$ forces
\begin{equation}
\Psi(3b_0\Lambda)\ge1-\varepsilon=\Psi(q_\varepsilon),
\end{equation}
and therefore $\Lambda\ge q_\varepsilon/(3b_0)$. The definition of $\Lambda_*(\varepsilon)$ restricts to $\Lambda\ge1$, so the lower bound in \eqref{eq:optimalTradeoff} follows. Since
\begin{equation}
\Phi^{-1}(1-u)\asymp\sqrt{\log(1/u)}
\qquad(u\downarrow0),
\end{equation}
the two bounds in \eqref{eq:optimalTradeoff} imply \eqref{eq:optimalTradeoffOrder}.

It remains to prove the exact profile \eqref{eq:optimalSprofile}. Let $K=-K$ be closed and convex with $\gamma_n(K)\ge1/2$. The $S$-inequality gives
\begin{equation}
\gamma_n(sK)\ge\Psi\bigl(s\Psi^{-1}(\gamma_n(K))\bigr)
\ge\Psi(sa_0),
\qquad s\ge1.
\end{equation}
This proves one inequality. The reverse inequality is attained by the symmetric strip
\begin{equation}
\label{eq:optimalityStrip}
K_0=\{x\in\R^n:|x_1|\le a_0\},
\end{equation}
for which $\gamma_n(K_0)=1/2$ and $\gamma_n(sK_0)=\Psi(sa_0)$.

The same value remains the infimum among compact convex bodies and finite polytopes. In dimension one the interval $[-a_0,a_0]$ already gives equality. Suppose $n\ge2$. For $R$ large enough that $\Psi(R)^{n-1}>1/2$, choose $a_R>0$ by
\begin{equation}
\label{eq:optimalityRectangleWidth}
\Psi(a_R)\Psi(R)^{n-1}=\frac12
\end{equation}
and define the centrally symmetric box
\begin{equation}
\label{eq:optimalityRectangles}
K_R=[-a_R,a_R]\times[-R,R]^{n-1}.
\end{equation}
Then $K_R$ is a compact centrally symmetric finite polytope and $\gamma_n(K_R)=1/2$. As $R\to\infty$, one has $a_R\to a_0$, and therefore
\begin{equation}
\gamma_n(sK_R)=\Psi(sa_R)\Psi(sR)^{n-1}\longrightarrow\Psi(sa_0).
\end{equation}
This proves the final assertion and completes the proof.
\end{proof}

Two features of the theorem are worth distinguishing. The one-dimensional interval \eqref{eq:optimalityInterval} proves optimality of the order $\sqrt{\log(1/\varepsilon)}$ for Talagrand's problem itself. The strip and rectangles prove the exact optimality of the function $\Psi(sa_0)$ only after one has reduced the available information to a symmetric convex core of measure $1/2$. Thus the exact threshold $\Phi^{-1}(7/8)/\Phi^{-1}(3/4)$ cannot be improved by a stronger use of the $S$-inequality alone, but it could in principle be improved by proving additional geometric information about the particular finite-barrier cores.

\section{A constructive high-measure ellipsoid}

We now prove \cref{thm:constructiveellipsoid}. The point is that Song's theorem gives the existence of an ellipsoid at the optimal dimension-dependent scale, while the finite half-space description supplied by \cref{thm:main} turns the remaining search into a semidefinite feasibility problem. We record the external input in the precise form needed below. By \cite[Corollary~0.7 and Section~3.2]{Song2026}, there is a universal integer $q_{\mathrm E}\ge1$ such that whenever $K\subset\R^n$ is closed and $\gamma_n(K)\ge2/3$, there exists a centered ellipsoid $E$ satisfying $\gamma_n(E)\ge1/2$ and
\begin{equation}
\label{eq:SongEllipsoid}
\sqrt{\frac{\log n}{n}}\,E\subset K^{(q_{\mathrm E})}.
\end{equation}
Moreover, Song's proof may be normalized so that
\begin{equation}
\label{eq:SongTraceNormalization}
E=E_Q:=\{x\in\R^n:x^TQx\le1\},\qquad Q\succeq0,\qquad \operatorname{Tr}Q=\frac1{10}.
\end{equation}
With this normalization,
\begin{equation}
\E(G^TQG)=\operatorname{Tr}Q=\frac1{10},
\end{equation}
so Markov's inequality gives $\gamma_n(E_Q)\ge9/10$, and in particular $\gamma_n(E_Q)\ge1/2$. Song also remarks that the ellipsoid arising from the argument is not explicitly constructed.

\begin{proof}[Proof of \cref{thm:constructiveellipsoid}]
Set
\begin{equation}
\label{eq:ellipsoidCompactCore}
r_\eta=\sqrt{\frac n\eta},
\qquad
S=A\cap(-A)\cap\clB{r_\eta}.
\end{equation}
Gaussian invariance under reflection and the union bound give
\begin{equation}
\gamma_n(A\cap(-A))\ge2\gamma_n(A)-1
\ge\frac23+2\eta.
\end{equation}
Since $\E|G|^2=n$, Markov's inequality gives $\Pp(|G|>r_\eta)\le\eta$, and hence
\begin{equation}
\label{eq:ellipsoidCoreMeasure}
\gamma_n(S)\ge\frac23+\eta>\frac23.
\end{equation}
The set $S$ is compact and centrally symmetric. Apply the first assertion of \cref{thm:main} with $\lambda=2$. This constructs a centrally symmetric finite polytope
\begin{equation}
\label{eq:ellipsoidK}
K\subset2(S+S+S),
\qquad
\gamma_n(K)\ge\frac12.
\end{equation}
Put $P=2K$. The $S$-inequality gives
\begin{equation}
\label{eq:ellipsoidPMeasure}
\gamma_n(P)\ge\Psi(2a_0)>\frac23,
\end{equation}
while
\begin{equation}
\label{eq:ellipsoidPContainment}
P\subset4(S+S+S)\subset4(A+A+A).
\end{equation}

Write a standard half-space representation of the constructed polytope as
\begin{equation}
\label{eq:CellipsoidHalfspaces}
P=\bigcap_{i=1}^M\{x\in\R^n:\ip{w_i}{x}\le b_i\},
\qquad b_i>0.
\end{equation}
The numbers $b_i$ are positive because $P$ is centrally symmetric, has positive Gaussian measure, and contains the origin in its interior. Apply Song's theorem to $P$. Since $P$ is convex,
\begin{equation}
P^{(q_{\mathrm E})}=q_{\mathrm E}P:
\end{equation}
the inclusion from left to right follows by averaging the summands, and the reverse inclusion follows by taking all summands equal. Therefore there exists an ellipsoid $E_Q$ of the form \eqref{eq:SongTraceNormalization} such that, for
\begin{equation}
\label{eq:ellipsoidrn}
r_n=q_{\mathrm E}^{-1}\sqrt{\frac{\log n}{n}},
\end{equation}
one has $r_nE_Q\subset P$. We now show that a matrix with this property can be recovered from the finitely many coefficients in \eqref{eq:CellipsoidHalfspaces}.

Consider the semidefinite feasibility problem
\begin{equation}
\label{eq:ellipsoidSDP}
\boxed{\begin{aligned}
\text{find}\quad &Q\in\mathbb S^n,\\
\text{subject to}\quad&\operatorname{Tr}Q=\frac1{10},\\
&\begin{pmatrix}Q&r_nw_i\\ r_nw_i^T&b_i^2\end{pmatrix}\succeq0,
\qquad i=1,\ldots,M.
\end{aligned}}
\end{equation}
This is a finite semidefinite feasibility problem with one symmetric matrix variable and $M$ linear matrix inequalities. It is feasible: the matrix supplied by Song's theorem satisfies the trace normalization, and the containment $r_nE_Q\subset P$ implies each of the constraints below, as we now verify.

Let $Q$ be any feasible solution of \eqref{eq:ellipsoidSDP}. Since $P$ is bounded, the normals $w_1,\ldots,w_M$ span $\R^n$; otherwise a nonzero vector orthogonal to all of them would generate a line contained in the polyhedron. We claim that $Q$ is positive definite. Suppose $z\in\ker Q$. For a positive semidefinite block matrix $\left(\begin{smallmatrix}Q&a\\a^T&b\end{smallmatrix}\right)$ one necessarily has $a\perp\ker Q$: otherwise the quadratic form evaluated on $(tz,s)$ with a suitable choice of $t$ would take negative values. Applying this to the $i$th constraint gives $\ip{w_i}{z}=0$ for every $i$. Since the normals span $\R^n$, this forces $z=0$. Thus $Q\succ0$.

For $Q\succ0$, the Schur complement shows that the $i$th block constraint in \eqref{eq:ellipsoidSDP} is equivalent to
\begin{equation}
\label{eq:ellipsoidSchur}
r_n^2w_i^TQ^{-1}w_i\le b_i^2.
\end{equation}
On the other hand, the support function of $E_Q$ in the direction $w_i$ is
\begin{equation}
\label{eq:ellipsoidSupport}
\sup_{x\in E_Q}\ip{w_i}{x}=\sqrt{w_i^TQ^{-1}w_i}.
\end{equation}
Indeed, after writing $x=Q^{-1/2}y$, the constraint becomes $|y|\le1$, and Cauchy--Schwarz gives the displayed maximum, with equality in the direction $Q^{-1/2}w_i$. Combining \eqref{eq:ellipsoidSchur} and \eqref{eq:ellipsoidSupport}, every $x\in E_Q$ satisfies $\ip{w_i}{r_nx}\le b_i$ for every $i$. Hence
\begin{equation}
\label{eq:ellipsoidInsideP}
r_nE_Q\subset P\subset4(A+A+A).
\end{equation}
The trace normalization in \eqref{eq:SongTraceNormalization} and Markov's inequality give $\gamma_n(E_Q)\ge1/2$.

Set
\begin{equation}
\label{eq:ellipsoidScaling}
t_\varepsilon=\frac{q_\varepsilon}{a_0}\ge1,
\qquad
E_\varepsilon=t_\varepsilon E_Q.
\end{equation}
The $S$-inequality applied to the centered ellipsoid $E_Q$ gives
\begin{equation}
\label{eq:ellipsoidHighMeasure}
\gamma_n(E_\varepsilon)
\ge\Psi(t_\varepsilon a_0)
=\Psi(q_\varepsilon)
=1-\varepsilon.
\end{equation}
From \eqref{eq:ellipsoidInsideP},
\begin{equation}
\label{eq:ellipsoidFinalContainment}
\frac{r_n}{4t_\varepsilon}E_\varepsilon
=\frac{a_0}{4q_{\mathrm E}q_\varepsilon}
  \sqrt{\frac{\log n}{n}}\,E_\varepsilon
\subset A+A+A.
\end{equation}
Thus \eqref{eq:ellipsoidMain} holds with $c=a_0/(4q_{\mathrm E})$. Every step before \eqref{eq:ellipsoidSDP} is one of the finite operations in \cref{thm:main}; solving that finite semidefinite system and applying the explicit scalar $t_\varepsilon$ completes the construction.
\end{proof}

Song shows that the order $\sqrt{(\log n)/n}$ cannot be improved by more than a universal constant even when the original set is convex \cite[Section~3.3.1]{Song2026}. The dependence on the target measure is also necessary. Fix $0<\eta<1/6$, put
\begin{equation}
\label{eq:ellipsoidSlabWidth}
b_\eta=\Phi^{-1}\left(\frac{11}{12}+\frac\eta2\right),
\end{equation}
and take the convex slab
\begin{equation}
\label{eq:ellipsoidSlabExample}
A_\eta=[-b_\eta,b_\eta]\times\R^{n-1}.
\end{equation}
Then $\gamma_n(A_\eta)=5/6+\eta$ and
\begin{equation}
A_\eta+A_\eta+A_\eta=[-3b_\eta,3b_\eta]\times\R^{n-1}.
\end{equation}
If $E$ is a centered ellipsoid with $\gamma_n(E)\ge1-\varepsilon$ and $\alpha E\subset A_\eta+A_\eta+A_\eta$, let $h_E(e_1)=\sup_{x\in E}|x_1|$. Since $E\subset\{|x_1|\le h_E(e_1)\}$,
\begin{equation}
1-\varepsilon\le\gamma_n(E)\le\Psi(h_E(e_1)),
\end{equation}
so $h_E(e_1)\ge q_\varepsilon$. The containment in the triple slab gives $\alpha h_E(e_1)\le3b_\eta$, and therefore
\begin{equation}
\alpha\le\frac{3b_\eta}{q_\varepsilon}.
\end{equation}
Thus the factor $q_\varepsilon^{-1}$ in \eqref{eq:ellipsoidMain} is optimal up to a constant depending only on the fixed input margin $\eta$. The two examples establish the two asserted optimality statements separately. The external theorem is used only to prove that the semidefinite system \eqref{eq:ellipsoidSDP} is feasible. Once the half-space description of $P$ has been constructed, the matrix $Q$ is obtained entirely from this finite data and no further access to the original set $A$ is required.

\section{Consequences and examples}

The conclusion of \cref{thm:main} should not be interpreted as saying that three Minkowski additions make a large set convex.  This can fail even for a compact subset of the line.

\medskip\noindent\textit{A nonconvex triple sum.} Let
\begin{equation}
\label{eq:exampleParameters}
     a=\Phi^{-1}\left(\frac{31}{32}\right), \qquad b=\Phi^{-1}\left(\frac{63}{64}\right), \qquad c=\Phi^{-1}\left(\frac{55}{64}+\Phi(-4)\right).
\end{equation}
For the compact set
\begin{equation}
\label{eq:exampleA}
     A=[-4,c]\cup[a,b],
\end{equation}
one has $\gamma_1(A)=7/8$ and
\begin{equation}
\label{eq:exampleTripleSum}
     A+A+A=[-12,c+2b]\cup[3a,3b].
\end{equation}
In particular, $A+A+A$ is not convex.
Indeed, the definition of $c$ gives $\gamma_1([-4,c])=55/64$, while $\gamma_1([a,b])=1/64$. Since $c<a$, the two intervals in \eqref{eq:exampleA} are disjoint and $\gamma_1(A)=7/8$. Put $I=[-4,c]$ and $J=[a,b]$. Then $A+A+A=3I\cup(2I+J)\cup(I+2J)\cup3J$, where $3I=[-12,3c]$, $2I+J=[a-8,2c+b]$, $I+2J=[2a-4,c+2b]$ and $3J=[3a,3b]$. The first three intervals overlap because $a-8<3c$ and $2a-4<2c+b$, and hence
\begin{equation}
\label{eq:firstThreeIntervals}
     3I\cup(2I+J)\cup(I+2J)=[-12,c+2b].
\end{equation}
On the other hand, $c\simeq1.078$, $a\simeq1.863$ and $b\simeq2.154$, so that
\begin{equation}
\label{eq:exampleGap}
     c+2b\simeq5.385<5.588\simeq3a.
\end{equation}
Thus $3J$ is separated from the interval in \eqref{eq:firstThreeIntervals}, which proves \eqref{eq:exampleTripleSum}.

The point of the example is that the convex set appearing in Talagrand's theorem cannot in general be taken to be the whole triple sum. Even at the threshold $\gamma_1(A)=7/8$, and even for a compact set in one dimension, the set $A+A+A$ may retain a genuine gap. The theorem asserts instead that an appropriate dilation of the triple sum contains a large convex subset.

\medskip\noindent\textit{Complete convexification.} The opposite behavior occurs in equally elementary examples. Let $n=2$, put $r=\sqrt{2\log(16/15)}$ and $R=\sqrt{2\log16}$, and consider the compact annulus $A=\{x\in\R^2:r\le |x|\le R\}$. Since $\Pp(|G|\le s)=1-e^{-s^2/2}$ for $G\sim\mathcal N(0,I_2)$, one has $\gamma_2(A)=7/8$. Nevertheless
\begin{equation}
\label{eq:annulusDoubleSum}
     A+A=\clB{2R}.
\end{equation}
Indeed, one inclusion follows from the triangle inequality, while the reverse inclusion follows because the circle $\{x:|x|=R\}$ lies in $A$ and every point of $\clB{2R}$ is the sum of two vectors of norm $R$. Consequently $A+A+A=\clB{3R}$. Thus Minkowski addition may destroy nonconvexity completely, while the preceding example shows that this behavior is not automatic.

\medskip\noindent\textit{Two summands are not sufficient.} Talagrand recalls in \cite{Talagrand2026} that for every $L>0$ one can find a dimension $N$ and a balanced set $A\subset\R^N$ with $\gamma_N(A)\ge3/4$ such that $L(A+A)$ does not contain a large convex set. Thus no universal dilation of the double sum can provide the type of dimension-free convexification appearing here. The occurrence of three Gaussian summands in the theorem of Hua, Song and Tudose therefore reflects a genuine geometric distinction.

\medskip\noindent\textit{Balanced sets and high-measure convexification.}

For balanced sets the symmetrization loss in \eqref{eq:defS} disappears. Recall that $A\subset\R^n$ is balanced if $tx\in A$ whenever $x\in A$ and $|t|\le1$. In particular $A=-A$ and $0\in A$. If $A$ is closed, balanced and $\gamma_n(A)\ge3/4$, then $S=A\cap\clB{\sqrt{24n}}$ is compact and symmetric, and the same Markov estimate as before gives $\gamma_n(S)\ge3/4-1/24=17/24>2/3$. The proof of the finite symmetric barrier uses the compact core only through compactness, symmetry and the strict inequality $\gamma_n(S)>2/3$. Hence the entire construction of \cref{thm:main} applies verbatim to balanced sets without intersecting with $-A$.

\begin{theorem}[Balanced sets]
\label{thm:balanced}
Let $A\subset\R^n$ be closed and balanced with $\gamma_n(A)\ge3/4$.
\begin{enumerate}[label=\textup{(\roman*)},leftmargin=2.2em]
\item One can construct a centrally symmetric finite polytope $C\subset A^{(6)}$ with $\gamma_n(C)\ge3/4$.
\item For every integer $m\ge2$ and every $p<\Psi(ma_0)$ one can construct a centrally symmetric finite polytope $C\subset A^{(3m)}$ with $\gamma_n(C)\ge p$.
\item If $N(\varepsilon)$ is the least universal number of summands sufficient to guarantee a convex subset of Gaussian measure $1-\varepsilon$, then $N(\varepsilon)=\Theta(\sqrt{\log(1/\varepsilon)})$ as $\varepsilon\downarrow0$.
\end{enumerate}
\end{theorem}

\begin{proof}
Run the construction in the proof of \cref{thm:main} with the compact balanced core $S=A\cap\clB{\sqrt{24n}}$. For every $\Lambda>1$ and every $p<\Psi(\Lambda a_0)$ it produces a centrally symmetric finite polytope $C\subset\Lambda(A+A+A)$ with $\gamma_n(C)\ge p$.

For (i), the numerical threshold for measure $3/4$ is $\Phi^{-1}(7/8)/a_0=1.705510543\ldots<2$. Choose $\Lambda$ strictly between this number and $2$. Then $3/4<\Psi(\Lambda a_0)$, so the construction gives $C\subset\Lambda(A+A+A)$ with $\gamma_n(C)\ge3/4$. It remains only to absorb the scalar dilation into additional Minkowski summands. If $a\in A$, balancedness and $\Lambda/2<1$ imply $(\Lambda/2)a\in A$. Therefore $\Lambda a=(\Lambda/2)a+(\Lambda/2)a\in A+A$. Applying this separately to each of the three terms in $A+A+A$ gives $\Lambda(A+A+A)\subset A^{(6)}$.

For (ii), fix an integer $m\ge2$ and a number $p<\Psi(ma_0)$. By continuity of $\Psi$, choose $\Lambda<m$ sufficiently close to $m$ that $p<\Psi(\Lambda a_0)$. Construct $C\subset\Lambda(A+A+A)$ with $\gamma_n(C)\ge p$. Since $\Lambda/m<1$, balancedness gives $(\Lambda/m)a\in A$ for every $a\in A$, and $\Lambda a$ is the sum of $m$ copies of this point. Hence $\Lambda A\subset A^{(m)}$, and consequently $\Lambda(A+A+A)\subset A^{(3m)}$.

For the upper bound in (iii), let $m_\varepsilon$ be the least integer $m\ge2$ satisfying $m a_0>\Phi^{-1}(1-\varepsilon/2)$. Then $1-\varepsilon<\Psi(m_\varepsilon a_0)$, and part (ii) produces a convex polytope of measure at least $1-\varepsilon$ inside $A^{(3m_\varepsilon)}$. Since $\Phi^{-1}(1-u)\asymp\sqrt{\log(1/u)}$ as $u\downarrow0$, this gives $N(\varepsilon)=O(\sqrt{\log(1/\varepsilon)})$.

For the converse, it is enough to work in dimension one. Let $b=\Phi^{-1}(7/8)$ and $A=[-b,b]$. Then $A$ is balanced and $\gamma_1(A)=3/4$, while $A^{(N)}=[-Nb,Nb]$. Every convex subset of this interval is again an interval contained in $[-Nb,Nb]$, so its Gaussian measure is at most $\Psi(Nb)$. If such a subset has measure at least $1-\varepsilon$, then $Nb\ge\Phi^{-1}(1-\varepsilon/2)$, and therefore
\begin{equation}
N\ge\frac{\Phi^{-1}(1-\varepsilon/2)}{\Phi^{-1}(7/8)}=\Omega\bigl(\sqrt{\log(1/\varepsilon)}\bigr).
\end{equation}
Together with the upper bound this proves (iii).
\end{proof}

The hierarchy is explicit. With $6$, $9$, $12$ and $15$ summands one obtains every Gaussian measure below $\Psi(2a_0)$, $\Psi(3a_0)$, $\Psi(4a_0)$ and $\Psi(5a_0)$, respectively; these numbers are approximately $0.8226$, $0.9569$, $0.9930$ and $0.99925$. For a prescribed $0<\varepsilon<1/2$, the choice $m_\varepsilon=\lfloor\Phi^{-1}(1-\varepsilon/2)/a_0\rfloor+1$ gives the explicit upper bound $N(\varepsilon)\le3m_\varepsilon$, while the one-dimensional argument above gives $N(\varepsilon)\ge\Phi^{-1}(1-\varepsilon/2)/\Phi^{-1}(7/8)$.

The six-summand conclusion does not use the universal number of Gaussian summands in the subgaussian representation theorem of Hua--Song--Tudose. It is a direct consequence of the three-Gaussians convex-order theorem, the symmetric finite construction, and the ability of a balanced set to absorb any scalar dilation smaller than $2$ into two copies. The following independent construction is numerically weaker, but we retain it because it establishes a different point: the finite-barrier method is not specific to convex order, and it can be run with a restricted class of test functions which detects another probabilistic structure.

\medskip\noindent\textit{An independent finitization through subgaussian tests.}
For a set $A\subset\R^n$ and an integer $m\ge1$, write $A^{(m)}=A+\cdots+A$, with $m$ summands. Recall that $A$ is balanced if $tx\in A$ whenever $x\in A$ and $|t|\le1$. We shall use the following consequence of \cite{HuaSongTudose2026}. There is a universal integer $Q$ such that every centered $1$-subgaussian random vector $Y$ in $\R^n$ can be written as
\begin{equation}
\label{eq:Qgaussians}
Y=G_1+\cdots+G_Q
\end{equation}
in distribution, where each $G_j$ is standard Gaussian. Here $1$-subgaussian means that
\begin{equation}
\label{eq:subgaussiandef}
\Pp(|\ip{Y}{\theta}|\ge u)\le2e^{-u^2/2} \qquad(\theta\in S^{n-1},\ u>0).
\end{equation}

\medskip \noindent \textit{1. A compact large set inside finitely many sums of $A$.} Let $A\subset\R^n$ be closed and balanced with $\gamma_n(A)\ge3/4$. By inner regularity choose a compact set $K\subset A$ with $\gamma_n(K)>2/3$, and let $A_0=\{tx:x\in K,\ |t|\le1\}$. Then $A_0$ is compact, balanced, contained in $A$, and $\gamma_n(A_0)>2/3$.

We first record that
\begin{equation}
\label{eq:halfballA0}
\frac12B_2^n\subset A_0+A_0.
\end{equation}
Indeed, let $|h|\le1/2$. The total variation distance between $N(0,I_n)$ and $N(h,I_n)$ is at most $|h|/2\le1/4$, and therefore $\gamma_n(h-A_0)\ge\gamma_n(A_0)-1/4>5/12$. It follows that $\gamma_n(A_0\cap(h-A_0))>2/3+5/12-1=1/12$. Thus there are $x,y\in A_0$ with $h=x+y$, which proves \eqref{eq:halfballA0}.

Let
\begin{equation}
\label{eq:tQ}
t_Q=\Phi^{-1}\left(1-\frac1{8Q}\right)-\Phi^{-1}\left(\frac23\right), \qquad k=\lceil2t_Q\rceil, \qquad m_0=2k+1.
\end{equation}
The Gaussian isoperimetric inequality \cite{SudakovTsirelson1974,Borell1975} gives
\begin{equation}
\gamma_n(B+tB_2^n) \ge \Phi\!\left(\Phi^{-1}(\gamma_n(B))+t\right) \qquad(t\ge0)
\end{equation}
for every measurable $B\subset\R^n$. Applying this with $B=A_0$ and using $\gamma_n(A_0)>2/3$ and the definition of $t_Q$ gives
\begin{equation}
\label{eq:Ssubgaussianmeasure}
\gamma_n(A_0+t_QB_2^n)\ge1-\frac1{8Q}.
\end{equation}
By \eqref{eq:halfballA0}, one has $(k/2)B_2^n\subset A_0^{(2k)}$, and since $t_Q\le k/2$,
\begin{equation}
\label{eq:SinsideA}
S:=A_0+t_QB_2^n\subset A_0^{(m_0)}\subset A^{(m_0)}.
\end{equation}
The set $S$ is compact and balanced. Fix from now on the barrier height $B=2$ and define
\begin{equation}
\label{eq:Dsubgaussian}
D=2\sqrt2\,S^{(Q)}.
\end{equation}
Since $S$ is balanced and $2\sqrt2<3$, one has $2\sqrt2 S\subset S^{(3)}$, and consequently
\begin{equation}
\label{eq:DsubgaussianinsideA}
D\subset S^{(3Q)}\subset A^{(N_0)}, \qquad N_0=3Qm_0.
\end{equation}
Thus $N_0$ is universal.

\medskip \noindent \textit{2. The subgaussian test class.} For $\theta\in S^{n-1}$ and $s>0$, define
\begin{equation}
\label{eq:psidef}
\psi_{\theta,s}(x)=\frac{\cosh\left(\frac{s}{2}\ip{\theta}{x}\right)-1}{e^{s^2/2}-1}.
\end{equation}
Each $\psi_{\theta,s}$ is even, nonnegative and convex, and $\psi_{\theta,s}(0)=0$. If $G$ is standard Gaussian, then
\begin{equation}
\label{eq:psiGaussian}
\E\psi_{\theta,s}(G)=\frac{e^{s^2/8}-1}{e^{s^2/2}-1}\le\frac14.
\end{equation}
Indeed, putting $u=s^2/2$, convexity of the exponential gives $e^{u/4}\le3/4+(1/4)e^u$.

The normalization in \eqref{eq:psidef} is chosen so that a uniform bound on these tests implies subgaussianity after a fixed scaling. Indeed, suppose that $X$ is symmetric and
\begin{equation}
\label{eq:psiBoundTwo}
\E\psi_{\theta,s}(X)\le2 \qquad(\theta\in S^{n-1},\ s>0).
\end{equation}
Then $\E\cosh((s/2)\ip{\theta}{X})\le1+2(e^{s^2/2}-1)\le e^{s^2}$. Let $Y=X/(2\sqrt2)$. Replacing $s$ by $u/\sqrt2$ gives $\E\cosh(u\ip{\theta}{Y})\le e^{u^2/2}$, and therefore, for $r,u>0$,
\begin{equation}
\Pp(|\ip{\theta}{Y}|\ge r)\le\frac{\E\cosh(u\ip{\theta}{Y})}{\cosh(ur)}\le2e^{u^2/2-ur}.
\end{equation}
Taking $u=r$ gives \eqref{eq:subgaussiandef}, while symmetry gives $\E Y=0$. Thus $X/(2\sqrt2)$ is centered $1$-subgaussian.

Let $\mathcal H=\operatorname{conv}_{\mathrm{fin}}\{\psi_{\theta,s}:\theta\in S^{n-1},\ s>0\}$ be the finite convex hull of the functions $\psi_{\theta,s}$. Every $f\in\mathcal H$ is even, nonnegative and convex, $f(0)=0$, and by \eqref{eq:psiGaussian}, $\E f(G)\le1/4$.

\begin{lemma}[Finite subgaussian barrier]
\label{lem:subgaussianbarrier}
Let $x_1,\ldots,x_m\in D^c$. Then there are affine functions $\ell_i(x)=\ip{u_i}{x}+v_i$ with $v_i\le0$ such that
\begin{equation}
\label{eq:qsubgaussian}
q(x)=\max(0,\ell_1(x),\ldots,\ell_m(x))
\end{equation}
satisfies $q(x_i)>2$ for every $i$ and $\E q(G)\le1/4$.
\end{lemma}

\begin{proof}
We first claim that there exists $f\in\mathcal H$ such that $f(x_i)>2$ for every $i$. Otherwise, finite-dimensional separation applied to $\{(f(x_1),\ldots,f(x_m)):f\in\mathcal H\}$ and $(2,\infty)^m$ gives numbers $p_i\ge0$, $\sum_i p_i=1$, such that
\begin{equation}
\label{eq:subgaussianseparation}
\sum_{i=1}^m p_i f(x_i)\le2 \qquad(f\in\mathcal H).
\end{equation}
Let $Y$ take the value $x_i$ with probability $p_i$, let $\varepsilon$ be an independent symmetric sign, and put $X=\varepsilon Y$. Since every $\psi_{\theta,s}$ is even, \eqref{eq:subgaussianseparation} gives $\E\psi_{\theta,s}(X)\le2$ for every $\theta,s$. The preceding calculation shows that $X/(2\sqrt2)$ is centered $1$-subgaussian. Hence \eqref{eq:Qgaussians} gives standard Gaussian vectors $G_1,\ldots,G_Q$ such that $X/(2\sqrt2)=G_1+\cdots+G_Q$. By \eqref{eq:Ssubgaussianmeasure}, $\Pp(G_1,\ldots,G_Q\in S)\ge1-\sum_{j=1}^Q\Pp(G_j\notin S)\ge1-Q/(8Q)=7/8>0$. On this event $X\in2\sqrt2 S^{(Q)}=D$. On the other hand $D$ is symmetric and every $x_i\notin D$, so $\Pp(X\in D)=0$, a contradiction. Thus the required $f$ exists.

For each $i$, choose $u_i\in\partial f(x_i)$ and let $\ell_i(x)=f(x_i)+\ip{u_i}{x-x_i}$. Then $\ell_i\le f$, $\ell_i(x_i)>2$, and $v_i:=\ell_i(0)\le f(0)=0$. Defining $q$ by \eqref{eq:qsubgaussian}, one has $0\le q\le f$, hence $\E q(G)\le1/4$, and $q(x_i)>2$ for every $i$.
\end{proof}

\medskip \noindent \textit{3. The finite construction.} We now repeat the discretization step with the constants appropriate to the subgaussian test class. Since $D$ is compact, choose $R>0$ such that $D\subset B(0,R-1)$. Let $Z\sim N(0,1)$, put $\beta_R=\E(Z-R)_+>0$ and $L=(4\beta_R)^{-1}$. We first check the analogue of the slope estimate. Suppose $q(x)=\max(0,\ip{u_1}{x}+v_1,\ldots,\ip{u_m}{x}+v_m)$ satisfies $v_i\le0$, $\ip{u_i}{x_i}+v_i\ge2$, $x_i\in\clB R$, and $\E q(G)\le1/4$. Fix $i$ and write $\sigma_i=|u_i|$, $e_i=u_i/|u_i|$ and $\tau_i=-v_i/\sigma_i$. The constraint at $x_i$ shows $u_i\ne0$ and $0\le\tau_i<\ip{e_i}{x_i}\le R$. Since $q$ dominates the positive part of the $i$th affine function,
\begin{equation}
\frac14\ge\E q(G)\ge\E(\ip{u_i}{G}+v_i)_+=\sigma_i\E(Z-\tau_i)_+\ge\sigma_i\beta_R.
\end{equation}
Thus $|u_i|\le(4\beta_R)^{-1}=L$ for every affine piece, and the resulting polyhedral function is globally $L$-Lipschitz.

Set $\delta=1/(2L)=2\beta_R$ and $h=\delta/\sqrt n$. As before, consider all axis-parallel grid cubes of side length $h$ which meet $\clB R$. There are only finitely many of them and each has Euclidean diameter $\delta$. For each such cube $Q$ set $\rho_Q=\max_{x\in Q\cap\clB R}\dist(x,D)$. The maximum exists by compactness. If $\rho_Q>0$, choose a maximizing point $x_Q$ and retain the cube; if $\rho_Q=0$, discard it. Relabel the retained points as $x_1,\ldots,x_m$. If $x\in\clB R\setminus D$ and $Q$ is a grid cube containing $x$, then $\rho_Q\ge\dist(x,D)>0$, so $Q$ was retained and its maximizing point satisfies $|x-x_Q|\le\delta$. Hence these points form a finite $\delta$-net of $\clB R\setminus D$. Since $D=2\sqrt2\,S^{(Q)}$, each distance query is again a compact finite-dimensional optimization over $S^Q$.

Apply \cref{lem:subgaussianbarrier} to this finite net and consider the convex program
\begin{equation}
\label{eq:subgaussianprogram}
\begin{aligned}
\text{minimize}\quad&\Theta(u,v):=\E\max(0,\ip{u_1}{G}+v_1,\ldots,\ip{u_m}{G}+v_m),\\
\text{subject to}\quad&v_i\le0,\qquad i=1,\ldots,m,\\
&\ip{u_i}{x_i}+v_i\ge2,\qquad i=1,\ldots,m.
\end{aligned}
\end{equation}
The finite subgaussian barrier lemma provides a feasible point with objective at most $1/4$, so $\inf\Theta\le1/4$. We also need attainment, because the final construction uses actual coefficients rather than an infimizing sequence. Restrict to the feasible sublevel $\mathcal K_{\rm sg}=\{(u,v):\Theta(u,v)\le1/4\}$. The slope estimate above gives $|u_i|\le L$. Since $|x_i|\le R$, the linear constraint gives $v_i\ge2-\ip{u_i}{x_i}\ge2-LR$, while feasibility gives $v_i\le0$. Thus all coefficients lie in a fixed bounded set.

The set $\mathcal K_{\rm sg}$ is closed. The linear constraints are closed, and if $(u^{(k)},v^{(k)})\to(u,v)$ in the bounded coefficient region, then the integrands defining $\Theta$ converge pointwise. The coefficient bounds dominate them by $C(1+|G|)$ for a finite constant $C=C(R,L)$, which is Gaussian integrable. Dominated convergence therefore shows that $\Theta$ is continuous on this region. Hence $\mathcal K_{\rm sg}$ is compact, and the minimum in \eqref{eq:subgaussianprogram} is attained.

Let $(u_i,v_i)$ be a minimizer and let $q$ be the associated polyhedral function. Then $\E q(G)\le1/4$, $q(x_i)\ge2$ at every net point, and $q$ is $L$-Lipschitz. If $x\in\clB R\setminus D$, choose $x_i$ with $|x-x_i|\le\delta$. Then $q(x)\ge q(x_i)-L|x-x_i|\ge2-L\delta=3/2$. For $|x|>R$, put $y=Rx/|x|$. Since $D\subset B(0,R-1)$, $y\notin D$ and the preceding estimate gives $q(y)\ge3/2$. Writing $y=\theta x+(1-\theta)0$ with $\theta=R/|x|$, convexity and $q(0)=0$ yield $q(y)\le\theta q(x)$, and hence $q(x)\ge q(y)/\theta\ge3/2$. We have proved
\begin{equation}
\label{eq:subgaussianglobalbarrier}
x\notin D\quad\Longrightarrow\quad q(x)\ge\frac32.
\end{equation}
Define $C=\{x:q(x)\le1\}$. Then $C$ is a finite polyhedron and \eqref{eq:subgaussianglobalbarrier} gives $C\subset D$. By \eqref{eq:DsubgaussianinsideA}, $D\subset A^{(N_0)}$, so $C\subset A^{(N_0)}$. Finally Markov's inequality, now used with the expectation level $1/4$, gives
\begin{equation}
\gamma_n(C)=1-\Pp(q(G)>1)\ge1-\E q(G)\ge\frac34.
\end{equation}
This completes the independent finite construction based on the subgaussian representation theorem.

The preceding argument proves, independently of \cref{thm:balanced}, that there is a universal integer $N_0$ such that every closed balanced set $A$ with $\gamma_n(A)\ge3/4$ admits a constructible finite polyhedron $C\subset A^{(N_0)}$ with $\gamma_n(C)\ge3/4$. The construction again consists of finitely many compact finite-dimensional optimization problems followed by one finite-dimensional convex minimization problem. More generally, separating at a height $B>1$ instead of $2$ and repeating the same calculation gives, for every $\varepsilon>0$, a universal number of summands and a constructible polyhedron of Gaussian measure at least $1-\varepsilon$. This latter dependence is not competitive with the optimal $\sqrt{\log(1/\varepsilon)}$ estimate in \cref{thm:balanced}; its purpose is to show that the finitization principle works for a genuinely different test class.

\section{The finite construction and further remarks}

We summarize the finite construction. Fix $(\Lambda,p)$ with $0<p<\Psi(\Lambda a_0)$ and choose $\lambda\in(1,\Lambda)$ satisfying $p<\Psi((\Lambda/\lambda)a_0)$. Once $S=A\cap(-A)\cap\clB{\sqrt{24n}}$ has been formed, the operations involving the original set are the compact maximizations which decide which grid cubes meet the complement of $D=\lambda(S+S+S)$. For a cube $Q$ meeting $\clB R$, one computes
\begin{equation}
\label{eq:distancequery}
\rho_Q=\max_{x\in Q\cap\clB R}\dist(x,D)=\max_{x\in Q\cap\clB R}\min_{s_1,s_2,s_3\in S}|x-\lambda(s_1+s_2+s_3)|.
\end{equation}
Both extrema are over compact finite-dimensional sets. If $\rho_Q>0$, one retains a maximizing point $x_Q$. After this finite list of points has been produced, all dependence on $A$ disappears and one solves the convex program \eqref{eq:convexProgram}. The output is the intersection of the finitely many symmetric slabs in \eqref{eq:Klambda}, followed by the scalar dilation $\Lambda/\lambda$.

The crude size bound \eqref{eq:facetbound} makes the finiteness explicit. With $R=3\lambda\sqrt{24n}+1$ and $\beta_R=\E(Z-R)_+$,
\begin{equation}
\label{eq:facetboundexplicit}
m\le\left(2+\frac{4R\sqrt n}{(\lambda-1)\beta_R}\right)^n.
\end{equation}
The convex program has $m(n+1)$ real variables and the output polytope has at most $2m$ facets. This estimate is extremely large because the construction uses a uniform Euclidean grid at the worst global Lipschitz scale. It is recorded only to make the finiteness explicit; the number of compact optimizations and the number of variables depend only on the ambient dimension and the chosen slack.

The half-space representation is convenient for the ellipsoid application, but it is not the only finite output available. One may enumerate the vertices of the bounded polytope and solve a compact feasibility problem over $S^3$ for each vertex. This supplements the finite $H$-representation with a finite $V$-representation and with triple-sum witnesses for the vertices. As noted above, the global barrier, rather than the vertex witnesses alone, proves that the whole polytope lies in the nonconvex triple sum.

For the ellipsoid application, once the coefficients $(w_i,b_i)$ of the polytope are known, the matrix $Q$ is obtained from the semidefinite feasibility problem \eqref{eq:ellipsoidSDP}. Thus the passage from a large Gaussian set to an optimal-scale ellipsoid consists of two finite convex-optimization stages after the compact distance problems involving the original set have been solved.

\medskip\noindent\textit{Further remarks.}

The construction separates three logically different ingredients. The three-Gaussians theorem is used only to exclude a finitely supported convex-order certificate. Gaussian integrability is then used only to bound the slopes of the supporting affine functions, thereby making a finite net possible. Finally, Kwapie\'n's mean--median comparison and the $S$-inequality are applied only after the finite symmetric polytope has already been constructed. This separation is useful because improvements at any one of the three stages would propagate to different aspects of the conclusion.

At the existential level \cref{thm:nonsymmetric} shows, without any symmetry assumption, that $A+A+A$ itself already contains a convex set of Gaussian measure at least $1/2$ as soon as $\gamma_n(A)>2/3$. The finite procedure, however, uses the positive gap $\lambda-1$ in an essential way when it passes from a finite family of exterior constraints to the entire complement. It would be interesting to know whether the endpoint half-measure core can also be obtained by a finite variational procedure without first enlarging the triple sum.

The $S$-inequality is sharp among all symmetric convex sets once only their initial Gaussian measure is fixed. Therefore the numerical constant $\Phi^{-1}(7/8)/\Phi^{-1}(3/4)$ cannot be improved by treating the constructed half-measure polytope as an arbitrary symmetric convex set. A better universal dilation by the same general route would have to use additional information about the specific barriers produced by finite separation, or prove that their unit sublevels have Gaussian measure uniformly larger than $1/2$.

The facet bound is surely far from optimal. The grid ignores both the geometry of $D$ and the information revealed by the successive supporting hyperplanes. An adaptive discretization, a cutting-plane procedure, or a direct separation of the complement might produce a much smaller finite description while preserving the same proof of correctness. The present estimate should therefore be viewed only as a proof of finiteness, not as evidence that the convex core intrinsically requires so many facets.

The independent subgaussian construction points in another direction. The finite-barrier mechanism does not depend on convex order as such: it requires a class of convex tests, a finite separation argument, a probabilistic interpretation of the dual certificate, and a structural theorem which excludes that certificate. Finding other test classes and representation theorems to which the same scheme applies may be useful in problems where an existence argument is presently known only through an infinite-dimensional duality principle.

\section*{Acknowledgments}

The author thanks Michel Talagrand for helpful discussions and comments.

\clearpage

\section*{Statements and Declarations}

\noindent\textbf{Funding.}
The author gratefully acknowledges support from the Chair ``Artificial Intelligence and Quantitative Methods for Finance'' of the Fondation du Risque, in partnership with Jump Solutions France.

\bibliographystyle{abbrvnat}
\bibliography{references}
\end{document}